\documentclass{amsart}
\usepackage{enumitem,kantlipsum}

\usepackage[utf8]{inputenc}
\usepackage{physics}
\usepackage{amssymb}
\usepackage{amsmath}
\usepackage{hyperref}
\usepackage{amsfonts}
\usepackage{amsthm}
\usepackage{mathtools}
\usepackage{tikz}
\usepackage{tikz-cd}
\usepackage{algorithmic}
\usepackage{asymptote}
\usepackage{url}
\usepackage{colonequals} 
\usepackage{amsmath,amssymb}
\usepackage{todonotes}

\makeatletter
\newsavebox\myboxA
\newsavebox\myboxB
\newlength\mylenA

\hypersetup{colorlinks=true}

\def\equationautorefname~#1\null{%
  Equation~(#1)\null
}

\newcounter{itemadded}
\usepackage[OT2,T1]{fontenc}
\DeclareSymbolFont{cyrletters}{OT2}{wncyr}{m}{n}
\DeclareMathSymbol{\Sha}{\mathalpha}{cyrletters}{"58}

\newtheorem{assumption}{Assumption}
\newtheorem{theorem}{Theorem}[section]
\newtheorem{proposition}[theorem]{Proposition}
\newtheorem{example}[theorem]{Example}
\newtheorem{lemma}[theorem]{Lemma}
\newtheorem{corollary}[theorem]{Corollary}

\newtheorem{definition}[theorem]{Definition}
\newtheorem{notation}[theorem]{Notation}

\newcommand{\Lconn}{L^{\text{conn}}}
\newcommand{\ep}{\varepsilon}

\newcommand{\Sh}{\text{Sh}}

\newcommand{\R}{\mathbb{R}}
\newcommand{\Q}{\mathbb{Q}}
\newcommand{\Z}{\mathbb{Z}}

\newcommand{\C}{\mathbb{C}}

\newcommand{\Qlb}{\overline{\Q}_l}
\newcommand{\Fr}{\text{Fr}}
\newcommand{\Gal}{\mathrm{Gal}}
\newcommand{\NS}{\text{NS}}
\newcommand{\diag}{\mathrm{diag}}

\newcommand{\isig}{\iota(\sigma)}

\newcommand{\Hom}{\mathrm{Hom}}
\newcommand{\End}{\mathrm{End}}
\newcommand{\GO}{\mathrm{GO}}
\newcommand{\SO}{\mathrm{SO}}

\newcommand{\GL}{\mathrm{GL}}

\newcommand{\tW}{\widetilde{W}}
\newcommand{\MT}{\mathrm{MT}}
\newcommand{\pp}{\mathfrak{p}}
\newcommand{\OO}{\mathfrak{o}}
\newcommand{\EE}{\mathcal{E}}
\newcommand{\Sym}{\mathfrak{S}}
\newcommand{\T}{\mathrm{T}}
\newcommand{\Fp}{\mathbb{F}_p}
\newcommand{\fpb}{\overline{\mathbb{F}}_p}

\newcommand{\DD}{\mathcal{D}}
\newcommand{\HH}{\mathrm{H}}

\newcommand{\Po}{\mathbb{P}^1}

\newtheorem{innercustomthm}{Theorem}
\newenvironment{customthm}[1]
  {\renewcommand\theinnercustomthm{#1}\innercustomthm}
  {\endinnercustomthm}

\newtheorem{innercustomcor}{Corollary}
\newenvironment{customcor}[1]
  {\renewcommand\theinnercustomcor{#1}\innercustomcor}
  {\endinnercustomcor}

\usepackage{enumitem,kantlipsum}

\usepackage{amsmath}
\usepackage{kbordermatrix}
\renewcommand{\kbldelim}{(}
\renewcommand{\kbrdelim}{)}

\begin{document}
\title{Density of $\mu$-ordinary Primes for K3 Surfaces}

\author{Trung Can$^{1}$, Yuxin Lin$^{2}$}

\email{ctttrung@hcmus.edu.vn}
\email{bennylin@umich.edu}

\address{[1] University of Science, Vietnam National University, Ho Chi Minh City}

\address{[2] University of Michigan, Ann Arbor}
\date{}
\maketitle

\begin{abstract}
By a result of Bogomolov and Zarhin, a K3 surface $X$ over a number field $L$ has ordinary reduction at a density 1 set of primes after passing to a finite extension of $L$. In this paper, we refine this result for non-CM K3 surfaces whose transcendental Hodge structure has endomorphism field $F$ abelian over $\mathbb{Q}$. We further assume a condition on the connected components of the $l$-adic monodromy group of $X$, and, when $F$ is totally real, a parity condition on the rank of the transcendental lattice over $F$. Under these assumptions, we prove that the set of primes of $L$ at which $X$ has
$\mu$-ordinary reduction has density $1$. As a corollary, the set of primes of $L$ at which $X$ has ordinary reduction has density $1/[FL:L]$. We include explicit examples of K3 surfaces satisfying these conditions.
\end{abstract}

\tableofcontents

\section{Introduction} 
\subsection{Density of ordinary primes}
Let $A$ be an abelian variety over a number field $L$. For all but finitely
many primes $\mathfrak{p}$ of $L$, the variety $A$ has good reduction at
$\mathfrak{p}$. The reduction $A_{\mathfrak{p}}$ is called \emph{ordinary} if
its Newton polygon has only the slopes $0$ and $1$. A conjecture generally attributed to Serre asserts that, after passing to a suitable
finite extension of $L$, the set of primes at which $A$ has ordinary reduction has
density~$1$.\footnote{Throughout, density means natural density.} The
conjecture is known for elliptic curves (Serre \cite{Serre1989}) and for abelian surfaces (Ogus \cite{Ogus1982}), but it remains open in general.

The same question can be asked for other classes of varieties. Let $X$ be a K3
surface over $L$ and let $\mathfrak{p}$ be a prime of good reduction. The
reduction $X_{\mathfrak{p}}$ is ordinary if its Newton polygon has slopes $0,1,2$. Bogomolov and Zarhin
\cite{bogomolovzarhin} proved the analog of Serre's conjecture for K3
surfaces: after a finite extension of $L$, the surface $X$ has ordinary
reduction at a set of primes of density~$1$.

Neither statement says which extension is needed or what the density of
ordinary primes is over $L$ itself. For an elliptic curve $E$, the density is
$1$ if all endomorphisms of $E_{\overline{L}}$ are defined over $L$, and $1/2$
otherwise. Sawin \cite{SAWIN2016} expressed the density of ordinary primes of
an abelian surface over $L$ in terms of the component group of its $l$-adic
monodromy group and showed that it always equals $1$, $1/2$, or $1/4$.
Cantoral-Farf\'an et al \cite{serre2} considered
absolutely simple abelian varieties whose geometric endomorphism algebra is a
CM field $F$ acting with simple signature. When $F/\mathbb{Q}$ is Galois, they compute the density of ordinary primes over $L$ explicitly. When $F/\mathbb{Q}$ is abelian, they show that the reduction is $\mu$-ordinary at a density 1 set of primes of $L$. In this paper, we prove results of the same kind for K3 surfaces.

\subsection{$\mu$-ordinary reduction and main results}
Let $X$ be a K3 surface over $L$ and fix an embedding
$L \hookrightarrow \mathbb{C}$. Let
$T(X)_{\mathbb{Q}} \subset H^2(X_{\mathbb{C}}, \mathbb{Q})$ denote the
transcendental part of the cohomology, and let
\[
  F := \End_{\mathrm{Hdg}}\bigl(T(X)_{\mathbb{Q}}\bigr)
\]
be its algebra of Hodge endomorphisms. By a theorem of Zarhin
\cite{Zarhin1983}, $F$ is a number field, and it is either totally real or CM.

Let $\mathfrak{p}$ be a prime of $L$ of good reduction, with residue
characteristic $p$. The action of $F$ constrains the Newton polygon of
$X_{\mathfrak{p}}$ in a way that depends on how $p$ decomposes in $F$. The
lowest Newton polygon compatible with this constraint is called the
\emph{$\mu$-ordinary} Newton polygon. For K3 surfaces it is the polygon of
height $d$, with slopes
$\bigl(1-\tfrac{1}{d},\, 1,\, 1+\tfrac{1}{d}\bigr)$,
where the integer $d$ is determined by the splitting of $p$ in
$F$ (see Section~\ref{sec:preliminaries}). In particular, if $p$ splits
completely in $F$, then $d = 1$ and the $\mu$-ordinary polygon
is the ordinary one.

At primes where the $\mu$-ordinary polygon is not ordinary, ordinary reduction
cannot occur, so $\mu$-ordinary reduction is the best one can hope for. This
leads to a refinement of Serre's question: does $X$ have $\mu$-ordinary
reduction at a density 1 set of primes of $L$, without any extension of the
base field?

Let $\Lconn$ denote the smallest extension of $L$ over which the
$l$-adic monodromy group $G_{X,l}$ of $X$ becomes connected. Our main result
answers the refined question under the following assumptions.

\begin{theorem}\label{Thm: theorem 2}
Let $X$ be a K3 surface over a number field $L$ without complex
multiplication, and let $F = \End_{\mathrm{Hdg}}(T(X)_{\mathbb{Q}})$. Assume
that
\begin{enumerate}
  \item $F/\mathbb{Q}$ is an abelian extension;
  \item $\Lconn \subseteq FL$;
  \item if $F$ is totally real, then $\dim_F T(X)_{\mathbb{Q}}$ is even.
\end{enumerate}
Then the set of primes of $L$ at which $X$ has $\mu$-ordinary reduction has
density~$1$.
\end{theorem}

The $\mu$-ordinary and ordinary Newton polygons coincide exactly at the primes
of $L$ lying over rational primes that split completely in $F$. Hence
Theorem~\ref{Thm: theorem 2} determines the density of ordinary primes over
$L$.

\begin{corollary}\label{cor:1}
Under the assumptions of Theorem~\ref{Thm: theorem 2}, the set of primes of
$L$ at which $X$ has ordinary reduction has density $1/[FL:L]$.
\end{corollary}

Corollary~\ref{cor:1} refines the theorem of Bogomolov and Zarhin in this
setting. It identifies $FL$ as an extension over which ordinary reduction
occurs at a density 1 set of primes, and it computes the density of ordinary
primes over $L$ itself.

In \cite{serre2}, the endomorphism field is assumed to be CM. In contrast,
Theorem~\ref{Thm: theorem 2} also covers the case where $F$ is totally real,
provided that $\dim_F T(X)_{\mathbb{Q}}$ is even.

For K3 surfaces, Sawin has proved in an unpublished
manuscript that $\mu$-ordinary reduction occurs at a set of primes of
density~$1$. The present work was carried out independently. Our argument
is accessible and self-contained; in particular, we computed the
relevant trace function explicitly on the relevant orthogonal groups; we learned of
Sawin's manuscript only after our results were obtained. As described in \cite[Remark~4.20]{serre2}, Sawin's argument applies to
arbitrary K3 surfaces over number fields and to Kuga--Satake abelian
varieties. 


\subsection{Strategy of the proof}
We follow the $\ell$-adic method of Serre, Katz--Ogus, Sawin, in the form developed by Cantoral-Farfán et al 
\cite{serre2}, and adapt it to K3 surfaces. The proof has three key steps.
\begin{enumerate}
  \item We construct a conjugacy-invariant algebraic function on $G_{X,l}$
  that detects $\mu$-ordinary reduction. Combined with the Weil bounds and
  Serre's Chebotarev density theorem, this reduces
  the density question to showing that the function is non-constant on every
  connected component of $G_{X,l}$.
  \item Using the assumptions that $F/\mathbb{Q}$ is abelian and
  $\Lconn \subseteq FL$, we realize the component group of
  $G_{X,l}$ as a subgroup of the Weyl group of the orthogonal group.
  \item An explicit matrix computation then shows that the function is
  non-constant on each connected component.
\end{enumerate}

\subsection{Organization of the paper}
Section~\ref{sec:preliminaries} collects background and notation on
$\mu$-ordinary Newton polygons, $l$-adic monodromy groups and
Mumford--Tate groups of K3 surfaces. In Section~\ref{refinement} we
construct the invariant that detects $\mu$-ordinary reduction and establish
its properties. Section~\ref{sec: orthogonal group} contains the proof of
Theorem~\ref{Thm: theorem 2}. Finally, in Example~\ref{ex: explicit K3 family}
we construct explicit families of K3 surfaces to which
Theorem~\ref{Thm: theorem 2} applies.

\section{Preliminaries} \label{sec:preliminaries}

\subsection{K3 Hodge structure and orthogonal Shimura datum}
Let $X$ be a K3 surface over a number field $L$. Let $\overline{L}$ denote an
algebraic closure of $L$. Fix an embedding \(L\hookrightarrow \mathbb C\).
The cohomology group
$H^2(X_{\mathbb C},\mathbb \Z)$
is of rank $22$ over $\Z$ with symmetric
bilinear form \(\Psi\) induced by Poincar\'e
duality. Let $\text{NS}(X_{\mathbb C})$
be the N\'eron--Severi group of $X_{\C}$ and define the
transcendental part by
\[
T(X):=\NS(X_{\mathbb C})^{\perp}
\subset H^2(X_{\mathbb C}, \Z),
\]
where the orthogonal complement is taken with respect to \(\Psi\). 

Let $V:=T(X)_{\mathbb Q}$, then $\Psi$ restricts to a symmetric bilinear form on $V$. The Hodge decomposition on \(H^2(X_{\mathbb C},\mathbb Q)\) induces a rational
Hodge structure on \(V\). We define $\EE(X)$ to be the endomorphism algebra of $T(X)_{\Q}$ preserving this Hodge structure. That is,
\[
\EE(X):=\operatorname{End}_{\mathrm{Hdg}}(T(X)_{\mathbb Q})
\]


Let $\mathbb S=\operatorname{Res}_{\mathbb C/\mathbb R}\mathbb G_m$ be the Deligne torus. Let $h: \mathbb{S} \to \GL(V_{\R})$ be the structure homomorphism of the rational polarized K3 Hodge structure on $(V,\Psi)$. Since the Hodge structure is polarized by \(\Psi\), the morphism \(h\) factors
through the group of orthogonal similitude $\GO(V,\Psi)_{\mathbb R}$. The triple $(V,\Psi,h)$ determines an orthogonal Shimura datum
$\DD=(\GO(V,\Psi),\mathcal {X})$
where $\mathcal{X}$ is the $\GO(V,\Psi)(\mathbb R)$-conjugacy class of $h$. 

\begin{assumption}\label{Assumption: A1 and A2}
Throughout the paper, we assume that:

\begin{enumerate}
    \item[(A1)] $F=\mathcal{E}(X)$ is isomorphic to an abelian field extension of $\Q$;
    \item[(A2)] $X$ does not have complex multiplication, and if $F$ is totally real, then we also assume that $\dim_F(T(X)_{\Q})$ is even.
\end{enumerate}
\end{assumption}
\begin{notation}\label{notation: m and n}
Given $X$ satisfying $(A1)$ and $(A2)$:
\begin{itemize}
\item Let $*$ denote the complex conjugation on $F$ and let $F_0$ be the fixed subfield of $F$ under $*$. Notice that if $F$ is totally real, then $*$ is the identity and $F_0=F$.
\item Let $m=[F_0:\Q]$ and $$
n=
\begin{cases}
\dim_F(T(X)_\Q), & \text{if $F$ is CM},\\[4pt]
\frac{1}{2}\dim_F(T(X)_\Q), & \text{if $F$ is totally real}.
\end{cases}
$$
\end{itemize}
Notice that $2mn=\text{rank}_\Z (T(X))$.
\end{notation}

Now let $F=\EE(X)$. Then \(F\) acts on \(V\), and we define the
\(F\)-linear orthogonal similitude group by
\[\GO_F(V,\Psi):=\operatorname{Cent}_{\GO(V,\Psi)}(F)\]
Equivalently, for every \(\mathbb Q\)-algebra \(R\),
\[
\GO_F(V,\Psi)(R)
=
\left\{
g\in \GL_{F\otimes_{\mathbb Q}R}(V\otimes_{\mathbb Q}R)
\ \middle|\ 
\exists c(g)\in R^\times
\text{ such that }
\Psi(gu,gv)=c(g)\Psi(u,v)
\text{ for all }u,v
\right\}.
\]

Let $\Sh(\DD)$ be the orthogonal Shimura variety attached to Shimura datum $\DD$. Under the Assumption \ref{Assumption: A1 and A2}, the reflex field $E$ is contained in $FL$. Since $\Sh(\DD)$ has a canonical model over its reflex field $E$, it is defined over $FL$. Hence, for every prime $\pp \in FL$, we can consider the reduction of $\Sh(\DD)$ at $\pp$.

In this paper, we focus on the set of primes in $L$ with the following properties:
\begin{definition}[The set $S$]\label{def: good primes}
Given $X/L$ a K3 surface, let $S$ be the set of primes $\pp$ in $L$ satisfying the following:
\begin{itemize}
    \item $X$ has good reduction at $\mathfrak{p}$;
    \item Let $p=\pp \cap \Z$. Then $p$ is unramified in $F/\Q$
    \item $\mathfrak{p}$ has degree 1.
\end{itemize}
\end{definition}

Now, let $p$ be a rational prime that is unramified in $F$. Since $F/\Q$ is abelian, we have a well-defined Frobenius element $\Fr_p$ in $\text{Gal}(F/\Q)$. 
\begin{notation}\label{notation: d and k}
Given such a prime $p$, we set:
\begin{itemize}
\item Let $d$ be the order of $\Fr_p$ in $\Gal(F/\Q)$. In particular, $d \mid [F:\Q]$.
\item Let $K$ be the fixed field of $F$ under $\Fr_p$. That is, $K=F^{\langle \Fr_p \rangle}$. In particular, $K$ is the largest subfield of $F$ for which $p$ splits completely.
\end{itemize}
Notice that $d=[F:K]$.
\end{notation}

\begin{definition}[Natural density]\label{def: density of primes}
For a number field $L/\Q$, let $A$,$B$ be sets of primes in $L$.
Then, we define the natural density of set $A$ in $B$ as follows:
\[\rho(A:B)=\lim_{x \to \infty} \frac{ \#\{\pp \in A, N^L_{\Q}\pp \leq x\}}{\# \{\pp \in B, N^L_{\Q}\pp \leq x\}}\]
In particular, if $B=\{\pp \in \text{Spec}(\mathcal{O}_L)\}$, then we define the density of $A$ as:
\[\rho(A):=\rho(A:B)\]
\end{definition}
Since primes of degree one have density $1$, the set $S$ has density $1$ in $L$.

For $\pp \in S$, let $p=\pp \cap \Z$. Then, $p$ is unramified in $E$ and hence is a prime of good reduction for $\Sh(\DD)$. Hence, we can consider the special fibre at $p$ given by $\Sh(\DD)_{\fpb}$.

\subsection{The $\mu$-ordinary Newton polygon of K3 surface with additional structure}\label{subsec: mu-ord NP}

Let $X/L$ be a $\text{K3}$ surface satisfying $(A1)$ and $(A2)$. Let $\pp \in S$ be a good prime for $X$ and let $p=\pp \cap \Z$. Let $\mathbb{F}_{\pp}$ be the residue field of $\pp$. By assumption on $S$, we have that $\mathbb{F}_{\pp}=\Fp$. Let $X_{\pp}$ be the reduction of $X$ mod $\pp$. Let $W(\Fp) \cong \Z_{p}$ be the ring of Witt vectors on $\Fp$.

Let $\varphi$ denote the Frobenius action on the crystalline
cohomology group $H^2_{cris}(X_{\pp}/W(\Fp))$. Then, the Newton polygon $\nu_{\pp}(X)$ is the multiset of the p-adic valuations of the eigenvalues of $\varphi$. It can equivalently be computed as the $p$-adic valuation of the eigenvalues of Frobenius acting on $H^2(X_{\overline{L}},\Q_l)$ for $l \neq p$. Given a Newton polygon $\nu_{\pp}$, we call $\lambda \in \nu_{\pp}$ a \textbf{slope} and the number of times $\lambda$ appears in $\nu_{\pp}$ its \textbf{multiplicity}.

When $X/L$ is a K3 surface, since $H^2_{cris}(X_{\pp}/W(\Fp))$ has rank $22$ over $W(\Fp)$, $\nu_{\pp}(X)$ has length $22$. Moreover, by \cite{liedtkelectures}, there exists an integer $h:=h(X) \in \{1,2,\dots,10\}$ such that $\nu_{\pp}(X)$ has slopes $(1-1/h, 1, 1+1/h)$
with multiplicities $(h,22-2h,h)$, or $\nu_{\pp}(X)$ has slope $1$ with multiplicity $22$. 

Recall that we can construct an orthogonal Shimura variety $\Sh(\DD)$
from $X$ with good reduction at $\pp\in S$. In
\cite[Theorem 4.2]{rapoportrichartz}, Rapoport and Richartz described
the lowest Newton polygon occurring in the special fibre
$\Sh(\DD)_{\fpb}$ in terms of the splitting behavior of $p$ in
$F/\Q$. In our case, $\Sh(\DD)$ parametrizes K3 Hodge structures on
$T(X)_{\Q}$, which has rank $2mn$. Hence the Newton polygons occurring
on the special fibre $\Sh(\DD)_{\fpb}$ record the Frobenius slopes
coming from the transcendental part of $H^2(X,\Z)$. The
N\'eron--Severi part has rank $22-2mn$ and contributes only slope $1$.
Therefore, for each Newton polygon occurring on
$\Sh(\DD)_{\fpb}$, we adjoin $22-2mn$ copies of slope $1$.

We define $\mu_{\pp}$ to be the lowest length-$22$ Newton polygon
obtained in this way. Thus, $\mu_{\pp}$ is the lowest Newton polygon
that can occur for a K3 surface with the prescribed additional
$F$-structure.

\begin{proposition}[$\mu$-ordinary Newton polygon for $K3$ surface] \label{prop: ord-slope}
Let $X/L$ be a K3 surface satisfying the assumptions (A1) and (A2).
Let $\pp$ be a prime in $S$ and $p=\pp \cap \Z$. Let $d$ be as defined in Notation \ref{notation: d and k}. 

Then, 
$\mu_{\mathfrak{p}}$ has slopes $1 -\frac{1}{d}, 1, 1+\frac{1}{d}$ with multiplicities $d, 22-2d, d$.
\end{proposition}
Notice that if $d=1$, equivalently if $p$ splits completely in $F$, then the $\mu$-ordinary Newton polygon coincides with the ordinary Newton polygon. 

\begin{proof}
Recall that $V=\T(X)_{\Q}$ and $F$ acts on $V$. Moreover, $\Fr_p$ has order $d$ in $\Gal(F/\Q)$. Hence, $\Hom(F,\C)$ splits into $[F:\Q]/d$ orbits, each of
cardinality $d$. Denote $\Fr_p$ by $\sigma_p$ and  $\tau \circ \sigma_p^{-1}$ by $\sigma_p \circ \tau$, the Frobenius orbit of $\tau \in \Hom(F,\C)$ is $\mathfrak{o}_{\tau}
=
\{\tau,\sigma_p \circ \tau,\ldots,\sigma_p^{d-1}\circ \tau\}$.

We can decompose $V_{\C}$ into $\tau$-eigenspaces:
\[V_{\C}=\bigoplus_{\tau \in \Hom(F,\C)}V_{\tau}\]

where
$\dim_{\C}V_{\tau}
=
\begin{cases}
n, & \text{if $F$ is CM},\\
2n, & \text{if $F$ is totally real}.
\end{cases}
$ On the other hand, we have the Hodge decomposition \[V_{\C}=V^{2,0}\oplus V^{1,1} \oplus V^{0,2}\]
By the analogue of \cite[Proposition 4.3]{li2019newton} for weight 2 Hodge structure, the slopes of the $\mu$-ordinary Newton polygons can be computed as follows. Write the Hodge slopes of $V_{\tau}$ in increasing order as $a_{\tau,1}\leq a_{\tau,2}\leq\cdots\leq a_{\tau,\dim(V_{\tau})}$, then the $j$-th $\mu$-ordinary Newton slope associated with
$\mathfrak{o}$ is
\[
\lambda_{\mathfrak{o},j}
=
\frac{1}{d}
\sum_{i=0}^{d-1}
a_{\sigma_p^i\circ\tau,j},
\qquad 1\leq j\leq \dim(V_{\tau}).
\]
And each slope $\lambda_{\mathfrak{o},j}$ occurs with multiplicity $d$.

We apply this formula to our setting. Let $\tau_0$ be the unique
embedding such that $V_{\tau_0}\cap V^{0,2}\neq 0$.
\begin{enumerate}
\item If $F$ is CM, then $\tau_0^*\neq\tau_0$. The
Hodge slopes of these two eigenspaces are
\[
V_{\tau_0}:\{0,1,\ldots,1\},
\qquad
V_{\tau_0^*}:\{1,\ldots,1,2\},
\]
while every other $V_{\tau}$ has Hodge slopes
$\{1,\ldots,1\}$.
\begin{itemize}
\item If $\tau_0^* \not \in \OO_{\tau_0}$. Then, the orbit $\mathfrak{o}_{\tau_0}$ contains one Hodge-slope
list $\{0,1,\ldots,1\}$ and $d-1$ lists $\{1,\ldots,1\}$. Hence
\[
\lambda_{\mathfrak{o}_{\tau_0},1}
=
\frac{0+(d-1)}{d}
=
1-\frac1d,
\]
and all the other slopes associated with $\mathfrak{o}_{\tau_0}$ are
equal to $1$. Similarly, the orbit $\mathfrak{o}_{\tau_0^*}$ contains
one list $\{1,\ldots,1,2\}$ and $d-1$ lists $\{1,\ldots,1\}$, so
\[
\lambda_{\mathfrak{o}_{\tau_0^*},n}
=
\frac{2+(d-1)}{d}
=
1+\frac1d,
\]
and all its other slopes are equal to $1$.

\item If $\tau_0^* \in \OO_{\tau_0}$. Then, this orbit contains the two lists
$\{0,1,\ldots,1\},
\{1,\ldots,1,2\}$
and $d-2$ lists $\{1,\ldots,1\}$. By the above computation, $\lambda_{\OO_{\tau_0},1}=1-\frac{1}{d}$,$\lambda_{\OO_{\tau_0},n}=1+\frac{1}{d}$, and all the other slopes are $1$.
\end{itemize}
Thus, in either case, the exceptional Newton slopes are $1-\frac1d$ and $1+\frac1d$
each with multiplicity $d$, while all remaining slopes are equal to
$1$.

\item Otherwise, we have
$\tau_0^*=\tau_0$, and the Hodge slopes of $V_{\tau_0}$ are $\{0,1,\ldots,1,2\}$, while every other eigenspace again has slopes $\{1,\ldots,1\}$. By Assumption $(A2)$, we have $n\geq2$, so the positions containing $0$ and $2$ are
distinct, so we have the same conclusion for the Newton slopes as in case 1.
\end{enumerate}
In summary, we see that the Newton polygon of $T(X)$ has slopes $1-\frac{1}{d},1,1+\frac{1}{d}$ with multiplicities $d,2mn-2d, d$. The $\NS(X)$ part contributes $22-2mn$ additional copies of
slope $1$. Therefore, the full $\mu$-ordinary Newton polygon has slopes $1-\frac{1}{d},1,1+\frac{1}{d}$ with multiplicities $d,22-2d, d$, as desired. 
\end{proof}

\begin{definition}\label{def: mu-ord prime}
Let $\pp \in S$. We say $\pp$ is a $\mu$-ordinary (resp. non-$\mu$-ordinary) prime if the Newton polygon of the reduction $X_{\pp}$ coincides with $\mu_{\pp}$ (resp. does not coincide with $\mu_{\pp}$). 
\end{definition}



\subsection{$l$-adic monodromy group and connected field}\label{subsec: l conn}

We adapt the setup from \cite[Section 2.3]{serre2} to fit the framework of our paper. 

Let $X/L$ be a K3 surface. We fix an embedding of $L$ into $\C$ and write $X_\C$ for the corresponding complex K3 surface. Let $X_{\overline{L}}$ denote the base change of $X$ to $\overline{L}$.
Fix a prime $l$ and consider the $l$-adic representation on the étale cohomology: 
$$\rho_{X,l}: \text{Gal}(\overline{L}/L) \to \GL(H^2(X_{\overline{L}},\Q_l)).$$
The $l$-adic monodromy group, denoted by $G_{X,l}$, is defined as the Zariski closure of the image of $\rho_{X,l}$ in $\GL(H^2(X_{\overline{L}},\Q_l))$. Let $G_{X,l}^\circ$ be the identity component of $G_{X,l}$ viewed as an algebraic group over $\Q_l$. Let $\pi_0(G_{X,l})=G_{X,l}/G_{X,l}^{\circ}$ be the group of connected components of $G_{X,l}$, which is a finite group. 
\begin{definition}\cite[Definition 2.4]{serre2}\label{def: Lconn}
$\Lconn$ is the field extension of $L$ such that $\rho^{-1}_{X,l}(G^{\circ}_{X,l}) = \Gal(\overline{L}
/\Lconn).$
\end{definition}
By construction, $\Lconn/L$ is finite Galois with $\Gal(\Lconn/L) \simeq \pi_0(G_{X,l})$. In particular, $\Lconn$ is independent of the choice of $l$.



We assume one further condition on $X/L$ in terms of $\Lconn$.
\begin{assumption}[A3]\label{Assumption: A3}
$\Lconn \subseteq FL$.
\end{assumption}
Under Assumptions $(A1)$ and $(A3)$, since $\Gal(FL/L) \xhookrightarrow{} \Gal(F/\Q)$ and $\Gal(FL/L)  \twoheadrightarrow  \Gal(\Lconn/L)$, we know that $\pi_0(G_{X,l})$ is a quotient of $\Gal(FL/L)$ which is abelian. Moreover, each connected component in $\pi_0(G_{X,l})$ can be represented by some $\sigma \in \Gal(FL/L)$.
\subsection{Mumford-Tate conjecture on K3 surface} Let $\MT_X$ be the Mumford--Tate group of $X$. Recall that
$\MT(V,\Psi)$ is the smallest algebraic subgroup of $\GO(V,\Psi)$
defined over $\Q$ whose real points contain the image of the Hodge
structure morphism $h:\mathbb S\longrightarrow \GO(V,\Psi)_{\R}$.

Recall that we have the decomposition 
$\HH^2(X,\Q) = \NS(X)_\Q \oplus \T(X)_\Q$ and we let $V=\T(X)_{\Q}$.
The action of \(\MT_X\) preserves both summands. On the
N\'eron--Severi part, it acts through the weight character, which is
already determined by its action on \(V\). Hence restriction to \(V\)
identifies $\MT_X$ with $\MT(V,\Psi)$.

From the definition of $\End_{\text{Hdg}}(T(X)_{\Q})=\mathcal{E}(X) = F$, we can further deduce that $\MT(V,\Psi) \subseteq \GO_F(V,\Psi)$. Since $\MT(V,\Psi)$ is connected, we know that in fact it is contained in the identity component $\GO_F(V,\Psi)^{\circ}$. 

We have the following result of Zarhin stating that this containment is actually an isomorphism.
\begin{theorem}[Zarhin, \cite{Zarhin1983}]
Let \((V,\Psi)\) be an irreducible polarized rational Hodge structure
of K3 type, and let $F=\End_{\mathrm{Hdg}}(V)$.
Then,
\[
\MT(V,\Psi)=\GO_F(V,\Psi)^{\circ}.
\]
\end{theorem}

On the other hand, the Mumford-Tate conjecture holds for K3 surface by the following result of Tankeev.
\begin{theorem}[Tankeev,\cite{Tankeev1991, Tankeev1995}]\label{thm: MT for k3}
There exists a canonical isomorphism
$$G_{X,l}^\circ \simeq \MT_X \otimes \Q_l.$$
\end{theorem}
Let $V_l := V \otimes \Q_l \subset H^2(X_{\overline{L}},\Q_l)$ together with the bilinear form $\Psi_l:=\Psi \otimes \Q_l$. The Galois action preserves the N\'eron--Severi part and hence also
preserves its orthogonal complement $V_l$. Moreover, the Galois
action on the Tate twist $\NS(X)\otimes\Q_l(1)$
has finite image. Therefore, on the identity component
$G_{X,l}^{\circ}$, the action on the untwisted N\'eron--Severi part is
through the cyclotomic character and is determined by the action on
$V_l$. Hence restriction to $V_l$ identifies
$G_{X,l}^{\circ}$ with its image in
$\GO(V_l,\Psi_l)$. Combining the Mumford--Tate conjecture with Zarhin's theorem, we obtain
the following.


\begin{corollary}\label{anothergroup}
$G_{X,l}^\circ \cong \GO_F(V_l,\Psi_l)^{\circ}$
\end{corollary}

\section{The $\mu$-ordinary invariant $a_{\pp}$} \label{refinement}
In this section, we adapt the strategy of \cite[Section 3]{serre2} to the second
cohomology of a \text{K3} surface. The main idea is to construct a
normalized Frobenius trace $a_{\pp}$ which detects whether a prime is
$\mu$-ordinary. The normalization removes the weight-two
contribution from Frobenius, so that at a non-$\mu$-ordinary
prime $\pp$, the resulting trace belongs to a finite set independent
of $p$.

Let $\chi$ denote the one-dimensional character of $G_{X,l}$
induced by the Galois action on the Tate twist $\Q_l(-1)$.
With our convention for Frobenius, for every degree-one prime
$\mathfrak p$ above $p$, we have $\chi(\Fr_{\mathfrak p})=p$.
For each prime $\mathfrak{p} \in S$ above a rational prime $p$, the Frobenius $\Fr_\mathfrak{p} \in \text{Gal}(\overline{L}/L)$ is well-defined up to conjugacy. By abuse of notation, let us also use $\Fr_\mathfrak{p}$ for its image under $\rho_{X,l}$. 

Recall that $d=[F:K]$, where $K$ is the largest subfield of $F$ over which $p$ splits completely. The trace on the algebraic representation $\wedge^d H^2(X_{\overline{L}},\Q_l) \otimes \chi^{-d}$ gives a well-defined invariant:
$$a_\mathfrak{p}:= \Tr(\Fr_{\mathfrak{p}} \mid \wedge^d H^2(X_{\overline{L}},\Q_l) \otimes \chi^{-d}).$$

\begin{proposition} \label{detect-mu-ord}
If $\mathfrak{p}$ is not $\mu$-ordinary, then $a_\mathfrak{p}$ is an integer in $[-\binom{22}{d},\binom{22}{d}]$.
\end{proposition}

\begin{proof}
The characteristic polynomial of $\Fr_\mathfrak{p}$ acting on $H^2(X_{\overline{L}},\Q_l)$ has the form
$$P(T):= T^{22}+c_1 T^{21} + \dots + c_{21} T+ c_{22}.$$
By Deligne's proof of the Weil conjectures
\cite{Deligne81,Deligne7172}, the action of $\Fr_{\pp}$ is semisimple,
and the coefficients $c_i$ are integers independent of $l$. Let us split $P(T)$ over $\overline{\Q}$ into
$$P(T)=(T-\alpha_1) \dots (T-\alpha_{22}).$$
From Weil's conjecture, the $\alpha_i$ are algebraic integers with Archimedean absolute value $p$ and $\mathfrak{q}$-adic units for any prime $\mathfrak{q}$ of $L$ away from $p$.

By \cite[Chapter 8]{Ogusbook}, the $p$-adic valuations of $\alpha_i$ are exactly the slopes of the Newton polygon $\nu_{\pp}(X)$.
The trace of $\Fr_\mathfrak{p}$ on $\wedge^d H^2(X_{\overline{L}},\Q_l)$ is given by
$$b:=\text{Tr}(\Fr_\mathfrak{p} \mid \wedge^d H^2(X_{\overline{L}},\Q_l)) = \sum\limits_{I \subset \{1,\dots,22\}, \#I = d} \prod_{i \in I} \alpha_i$$
Moreover, $b$ is an integer as it can be written as a polynomial with integer coefficients in the variables $c_i$. Since $\mathfrak{p}$ is not $\mu$-ordinary, by Proposition \ref{prop: ord-slope}, the smallest slope of $\nu_{\pp}(X)$ is greater than $1-\frac{1}{d}$. Thus, the sum of any $d$ slopes of $\nu_{\pp}(X)$ is greater than $d-1$, i.e., the $p$-valuation of $\prod_{i \in I} \alpha_i$ is greater than $d-1$. Hence, the $p$-valuation of $b$ must be at least $d$, or equivalently $p^d \mid b$. In addition, as each $\alpha_i$ has absolute value $p$, we deduce that
$$ -p^d \binom{22}{d} \le b \le p^d \binom{22}{d}.$$
Since $\chi(\Fr_{\pp})=p$, we have that $a_{\pp}=\frac{b}{p^d}$. Hence $a_{\pp}$ is an integer in $[-\binom{22}{d}, \binom{22}{d}]$
\end{proof}

We now relate the density of non-$\mu$-ordinary primes to the behavior of
certain algebraic trace functions on the connected components of $G_{X,l}$. We will be using the following commutative diagram:
\[
\begin{tikzcd}
    G_{X,l}^\circ \arrow[d, hook] &  \text{Gal}(\overline{L}/L) \arrow[ld,"\rho_{X,l}"] \ni \Fr_\mathfrak{p}\arrow[d, two heads] \\
    G_{X,l} \arrow[d, two heads, "\pi_0"] & \text{Gal}(FL/L) \ni \sigma \arrow[d, two heads, "\kappa"] \arrow[r, hook, "\iota"] & \text{Gal}(F/\Q) \\
    \pi_0(G_{X,l}) \arrow[r, leftarrow, "\simeq"', "\rho_{X,l}"]  & \text{Gal}(\Lconn/L).
\end{tikzcd}
\]

Here  $\kappa$ is induced by the
inclusion $\Lconn\subseteq FL$, and $\iota$ is restriction to $F$. By
Assumption~\textup{(A3)}, the map $\kappa$ is surjective, while $\iota$ is
injective.


Let $S$ be the set of primes in $L$ in Definition \ref{def: good primes}. Let $S[\nu]:=\{\pp \in S: \pp \text{ is not $\mu$-ordinary}\}$. For $A,B$ sets of primes in $L$, let $\rho(A:B)$ be the density of $A$ in $B$ and $\rho(A)$ be the density of $A$ as in Definition \ref{def: density of primes}. 

We want to compute $\rho(S[\nu])$. Equivalently, since $\rho(S)=1$, we want to compute $\rho(S[\nu]:S)$. For this purpose, we consider the following partition of $S$:
\[
S
=
\coprod_{\sigma\in\Gal(FL/L)}S_\sigma.
\]
where $S_\sigma
:=
\left\{
\pp\in S:
\left.\Fr_\pp\right|_{FL}=\sigma
\right\}.$
Since $\Gal(FL/L)$ is abelian, this condition is independent of the choice
of the Frobenius element $\Fr_\pp$. 

By the Chebotarev density theorem, we have that $\rho(S_\sigma:S)
=
\frac{1}{[FL:L]}$
for every $\sigma\in\Gal(FL/L)$. Hence, it suffices to compute $\rho(S_{\sigma}[\nu]: S_{\sigma})$, where $S_{\sigma}[\nu]=S_{\sigma} \cap S[\nu]$.

Set $\Gamma
:=
\rho_{X,l}\bigl(\Gal(\overline L/FL)\bigr)$. 
Fix $\sigma\in\Gal(FL/L)$, let $G_{X,l}^{(\sigma)}$ be the connected component of $G_{X,l}$ corresponding to $\sigma$. For the same $\sigma$, choose a lift $\widetilde{\sigma}\in\Gal(\overline L/L)$
and define $\Gamma_\sigma
:=
\rho_{X,l}(\widetilde{\sigma})\Gamma$. Because $\Lconn\subseteq FL$, the subgroup $\Gamma$ is Zariski dense in $G_{X,l}^{\circ}$. Hence, in this notation, we have:
\begin{equation}
\overline{\Gamma}=G_{X,l}^{\circ} \qquad{} \overline{\Gamma}_{\sigma}=G_{X,l}^{(\sigma)}
\end{equation}

\begin{lemma}[Chebotarev--Serre]\label{lem:chebotarev-haar}
Let $\sigma\in\Gal(FL/L)$, let $\Gamma_{\sigma}$ be as defined above and let $\mu_{\sigma}$ be the translated Haar measure on $\Gamma_{\sigma}$. Let $A\subseteq\Gamma_\sigma$
be a conjugacy invariant Borel subset whose
boundary has $\mu_\sigma$-measure zero. Then
\[
\rho\left(
\left\{
\pp\in S_\sigma:
\rho_{X,l}(\Fr_\pp)\in A
\right\}
:S_\sigma
\right)
=
\mu_\sigma(A).
\]
\end{lemma}

\begin{proof}
This is the Chebotarev density theorem in the $l$-adic setting; see
\cite[Corollary~6.10]{Serrebook}.
\end{proof}

For $\sigma\in\Gal(FL/L)$, for every $\pp\in S_\sigma$, we have $\Fr_{\pp} \mid_F=\iota(\sigma)$. Hence
$d=\operatorname{ord}\bigl(\iota(\sigma)\bigr)$, which is  constant as $\pp$ varies in $S_\sigma$. Hence, we can define the conjugacy invariant algebraic function $b_{\sigma}$ on $G_{X,l}^{(\sigma)}$ as follows:
$$b_{\sigma}:= \Tr( - \mid \wedge^d H^2(X_{\overline{L}},\Q_l) \otimes \chi ^{-d}):G^{(\sigma)}_{X,l} \to \Q_l.$$
In particular, for every $\pp\in S_\sigma$, one has
$b_\sigma\bigl(\rho_{X,l}(\Fr_\pp)\bigr)
=a_\pp.$

\begin{proposition}\label{prop:density-bound}
The density of the set of non-$\mu$-ordinary primes satisfies
\[
\rho(S[\nu]:S)
\le
\frac{
\#\left\{
\sigma\in\Gal(FL/L):
b_\sigma\text{ is constant on }G_{X,l}^{(\sigma)}
\right\}
}{
[FL:L]
}.
\]
In particular, if $b_\sigma$ is nonconstant on $G_{X,l}^{(\sigma)}$ for
every $\sigma\in\Gal(FL/L)$, then the set of non-$\mu$-ordinary primes
has density zero.
\end{proposition}

\begin{proof}
We have that $\rho(S[\nu]:S)=\sum_{\sigma \in \Gal(FL/L)}\rho(S_{\sigma}[\nu]:S_{\sigma})
\rho(S_{\sigma}:S)=\sum_{\sigma \in \Gal(FL/L)}\frac{\rho(S_{\sigma}[\nu]:S_{\sigma})}{[FL:L]}$. Hence, we need to show that

\[
\rho\bigl(S_\sigma[\nu]:S_\sigma\bigr)
\le
\begin{cases}
1,
&
\text{ if $b_\sigma$ is constant on }G_{X,l}^{(\sigma)},\\[1mm]
0,
&
\text{ if $b_\sigma$ is nonconstant on }G_{X,l}^{(\sigma)}.
\end{cases}
\]
Fix $\sigma\in\Gal(FL/L)$ and set $C=
\binom{22}{d}$.
Consider the conjugacy-invariant closed subset
\[
Z'_\sigma
:=
\bigcup_{n\in[-C,C]\cap\Z}
b_\sigma^{-1}(n)
\subseteq
G_{X,l}^{(\sigma)}.
\]
And let $Z_{\sigma}=Z_{\sigma}' \cap \Gamma_{\sigma}$.
By Proposition~\ref{detect-mu-ord}, for every $\pp\in S_\sigma[\nu]$, we have $\rho_{X,l}(\Fr_{\pp}) \in  Z_\sigma$. Since $Z_\sigma$ is $l$-adic analytic, we know that the boundary of $Z_\sigma$ has $\mu_{\sigma}$ measure zero. Hence,
by Lemma~\ref{lem:chebotarev-haar} we have 
$\rho\bigl(S_\sigma[\nu]:S_\sigma\bigr) \leq
\mu_\sigma\bigl(Z_\sigma\bigr)$. Since $\Gamma_{\sigma}$ is Zariski dense in $G_{X,l}^{(\sigma)}$, we have that $\mu_{\sigma}(Z_{\sigma})=\begin{cases}
1 &\text{ if $G_{X,l}^{(\sigma)} \subseteq Z'_{\sigma}$; } \\
0 &\text{ otherwise.} \\
\end{cases}$

Now we want to show that $G_{X,l}^{(\sigma)} \subseteq Z'_{\sigma}$ if and only if $b_{\sigma}$ is constant on $G_{X,l}^{(\sigma)}$. 

Suppose $G_{X,l}^{(\sigma)} \subseteq Z'_{\sigma}$. Since $G_{X,l}^{(\sigma)}$ is connected, there exists $n \in [-C,C]$ such that $G_{X,l}^{(\sigma)} \subseteq b_{\sigma}^{-1}(n)$. Hence, $b_{\sigma}$ is constant. 

Conversely, if $b_\sigma$ is constant on $G_{X,l}^{(\sigma)}$, with 
constant value $c$. It suffices to show that $c \in [-C,C] \cap \Z$. Write $c=\frac{x}{y}$ such that $\gcd(x,y)=1.$  By definition of $b_{\sigma}$, we have that $c= \frac{b}{p^d}$ where recall that $b=\text{Tr}(\Fr_\mathfrak{p} \mid \wedge^d H^2(X_{\overline{L}},\Q_l))$. Hence, $y \mid p^d$, and this holds for every $\pp\in S_\sigma$ and $p=\pp \cap \Z$. Therefore, since $S_{\sigma}$ has positive density, we can find $\pp, \pp' \in S_{\sigma}$ such that $p \neq p'$. Hence $y \mid p^d$ and $y \mid p'^d$ for two distinct primes $p,p'$. This implies that $y=1$, hence $c$ is an integer. Moreover, the Weil bound $|b| \leq p^d C$ gives  $c \in [-C,C] \cap \Z$. This finishes the proof.
\end{proof}

\section{Computation on the orthogonal group}\label{sec: orthogonal group}
In this section, we fix a connected component $G_{X,l}^{(\sigma)}$ represented by $\sigma \in \Gal(FL/L)$ and show that the trace function
$b_\sigma$ is non-constant. By Proposition
\ref{prop:density-bound}, this will imply that the set of
non-$\mu$-ordinary primes has density zero in $L$. To prove the nonconstancy of
$b_\sigma$, we realize $G_{X,l}^{(\sigma)}$ as an explicit coset of matrices
and compute $b_\sigma$ on this coset.

Recall from Corollary \ref{anothergroup} that
$G_{X,l}^{\circ} \simeq \GO_F(V_l,\Psi_l)^{\circ}$. Since extending the coefficient field does not affect whether $b_\sigma$
is constant, we work over $\overline{\Q}_l$ for the remainder of this
section. Set
$$
G:=\text{O}(V_l,\Psi_l)\otimes \overline{\Q}_l,
\qquad
G_F:=
\ker\left(
\alpha:\GO_F(V_l,\Psi_l)^{\circ}\to \Q_l^{\times}\right)
\otimes \overline{\Q}_l
$$
where $\alpha$ is the orthogonal similitude character. In particular,
$G_F\subseteq G_{X,l}^{\circ}\otimes \overline{\Q}_l$.

If $B_\sigma$ is a representative of the connected component
$G_{X,l}^{(\sigma)}$, then
$B_\sigma G_F\subseteq G_{X,l}^{(\sigma)}$.
Hence, it suffices to prove that $b_\sigma$ is nonconstant on $B_\sigma G_F$. This is analogous to the restriction to the
similitude-one subgroup in the proof of \cite[Theorem 4.19]{serre2}.


Recall from Notation \ref{notation: m and n} that
$m=[F_0:\Q]$ and $2mn=\dim_\Q(V)$.

When $F$ is CM, then $[F:\Q]=2m$. For each
$\sigma_i \in \Hom(F_0, \Qlb)$, let
$\tau_i,\tau_i^*\in \Hom(F, \Qlb)$ be the two extensions of $\sigma_i$. The $F$-action gives a decomposition
$$
V_l\otimes_{\Q_l}\overline{\Q}_l
=
\bigoplus_{i=1}^{m}
\left(
V_l^{\tau_i}\oplus V_l^{\tau_i^*}
\right),
$$
where each $V_l^{\tau_i}$ and $V_l^{\tau_i^*}$ has dimension $n$. Since the pairing $\Psi_l$ is $F$-linear, we see that $V_l^\tau$ pairs nontrivially only with
$V_l^{\tau^*}$. Hence, after choosing
suitable bases, for each $1\leq i\leq m$, we may write
$
\Psi_l\big|_{V_l^{\tau_i}\oplus V_l^{\tau_i^*}}
=
\begin{pmatrix}
0 & J_i\\
J_i & 0
\end{pmatrix},
$ where $J_i$ is some invertible diagonal matrix.

When $F$ is totally real, then
$$
V_l\otimes_{\Q_l}\overline{\Q}_l
=
\bigoplus_{\tau\in \Hom(F, \Qlb)}V_l^\tau,
$$
where each $V_l^\tau$ has dimension $2n$. In this case, the character
spaces are mutually orthogonal with respect to $\Psi_l$, and the
restriction
$\Psi_l\big|_{V_l^\tau}
$ is a nondegenerate symmetric form on $V_l^\tau$.

The preceding decompositions give the following block matrix description
of $G_F$.

\begin{lemma}\label{lem: block matrix form of GF}
After passing to $\overline{\Q}_l$, the following hold.
\begin{enumerate}
    \item If $F$ is CM, then
    $$
    G_F\simeq \prod_{i=1}^{m}\GL_n
    $$
    where the isomorphism is given by
    $
    \diag \left(M_1,M_1^*,\ldots,M_{m},M_{m}^*
    \right) \longmapsto(M_1,\ldots,M_{m})$
    with
    $
    M_i^*
    :=
    J_i^{-1}(M_i^{-1})^\top J_i$.

    \item If $F$ is totally real, then
    $$G_F\simeq
    \prod_{i=1}^{m}\SO_{2n}
    $$
    where the isomorphism is given by
    $
    \diag(M_1,\ldots,M_m)\longmapsto(M_1,\ldots,M_m)
    ,
    $
    with
    $
    M_i^\top\Psi_l^{\tau_i}M_i
    =
    \Psi_l^{\tau_i}
    $.
\end{enumerate}
\end{lemma}
Next we consider the Weyl group of $G_F$ and $G$.
Let $T$ be the maximal torus of $G_F$ given by
$$
T=\left\{\diag(T_1,T_1^{-1},\dots,T_m,T_m^{-1}) \,\middle|\, \text{$T_i$ is diagonal of size $n$}\right\}.
$$
Since $T$ has rank $nm$, it is also a maximal torus of $G$.

We write $W_F$ for the Weyl group of $G_F$ with respect to $T$, and $\widetilde{W}$ for the extended Weyl group of $G$ with respect to $T$. That is,
$$
\tW:=N_G(T)/T, \qquad W_F:=N_{G_F}(T)/T
$$
We also set
$$
H:=(N_G(G_F)\cap N_G(T))/T.
$$
We have the following diagram for the inclusion relation among these various groups:
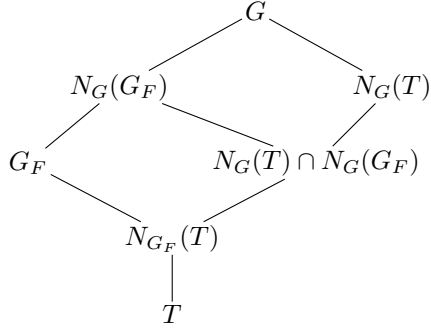
\begin{figure}[h]
\centering
\begin{tikzpicture}[every node/.style={inner sep=1pt}]
\node (G) at (0,0) {$G$};

\node (NGGF) at (-1.8,-1) {$N_G(G_F)$};
\node (NGT) at (1.8,-1) {$N_G(T)$};

\node (GF) at (-3,-2) {$G_F$};
\node (I) at (0.8,-2) {$N_G(T)\cap N_G(G_F)$};

\node (NGFT) at (-1.1,-3) {$N_{G_F}(T)$};

\node (T) at (-1.1,-4) {$T$};

\draw (G) -- (NGGF);
\draw (G) -- (NGT);
\draw (NGGF) -- (GF);
\draw (NGGF) -- (I);
\draw (NGT) -- (I);
\draw (GF) -- (NGFT);
\draw (I) -- (NGFT);
\draw (NGFT) -- (T);
\end{tikzpicture}
\label{fig:normalizer-diagram}
\end{figure}

In particular, we have 
\[N_{G_F}(T) \subseteq N_G(T) \cap N_G(G_F) \subseteq N_G(T)\] Passing to the quotients by $T$ we get that $W_F \subseteq H \subseteq \tW$.

The extended Weyl group $\tW$ can be realized as a subgroup of $\mathfrak{S}_{2mn}$. More precisely, we identify
$$
\mathfrak{S}_{2mn}=\Sym\{\pm 1,\pm 2,\dots,\pm mn\},
$$
where the subset $\{\pm(in-n+1),\dots,\pm in\}$ corresponds to the pair $(T_i,T_i^{-1})$. Then
\begin{equation}\label{eq: Weyl G}
\tW=\left\{\pi \in \Sym\{\pm 1,\dots,\pm mn\} \,\middle|\,
\pi(-i)=-\pi(i)\right\}.
\end{equation}
And $W_F$ can be identified as a subgroup of $\tW$. By Lemma \ref{lem: block matrix form of GF},
\begin{itemize}
\item if $F$ is CM, then $\prod_{i=1}^m \mathfrak{S}_n \cong W_F$, with the isomorphism given by
$(\pi_1,\dots,\pi_m)\mapsto (\pi_1,\pi_1,\dots,\pi_m,\pi_m)$.
\item if $F$ is totally real, then
$\prod_{i=1}^m W_{\SO_{2n}} \cong W_F$.
\end{itemize}

The next proposition is an adaptation of \cite[Proposition 4.14]{serre2}.

\begin{proposition}\label{prop: quotient map}
We have a natural isomorphism $$N_{G}(G_F)/G_F \simeq H/W_F.$$
\end{proposition}

\begin{proof}
We define a map from $N_{G}(G_F)$ to $H/W_{F}$ whose kernel is $G_F$. Let $g \in N_{G}(G_F)$, and consider $g T g^{-1}$, which is a maximal torus of $G_F$. Since maximal tori are conjugate, there exists $g_0 \in G_F$ such that $g T g^{-1} = g_0 T g_0^{-1}$, where $g_0$ is unique up to $N_{G_F}(T)$. It follows that $g_0^{-1} g$ is in $N_{G}(G_F) \cap N_{G}(T)$, so the map 
\begin{align*}
\begin{split}
N_{G}(G_F) & \to H/W_F \\
g &\mapsto g_0^{-1} g W_F
\end{split}
\end{align*}
 is well-defined. It is surjective since $N_{G}(G_F) \cap N_{G}(T) \subseteq N_{G}(G_F)$. The kernel of this map consists of $g$ such that $g_0^{-1} g \in N_{G_F}(T)$, or $g \in G_F.$
\end{proof}

The next lemma is an adaptation of \cite[Lemma 4.15]{serre2}.
\begin{lemma}\label{lem: H q Wf iso}
We have an isomorphism $$H/W_{F} \simeq \{\pm 1\}^m \rtimes \mathfrak{S}_m.$$
\end{lemma}

\begin{proof}
    \begin{enumerate}
        \item If $F$ is CM,  then by Lemma \ref{lem: block matrix form of GF}, $G_F \cong \prod_{m}\GL_n$, where the $i$-th copy of $\GL_n$ corresponds to the $i$-th block $(M_i,M_i^*)$ in $G_F$. Since $H$ normalizes $G_F$, conjugation by $h \in H$ permutes the blocks of $G_F$ and possibly switches the pairs. Hence, $h \in H$ can be represented by $h=(\ep_1, \dots, \ep_m;\tau;\pi_1, \dots, \pi_m)$, where $\ep_i \in \{\pm 1\}, \pi_i \in \Sym_n, \tau \in \Sym_m$. Since $W_{F}=\prod_m \Sym_n$, we get that $H/W_{F}=\{(\ep_1, \dots, \ep_m;\tau) \mid \ep_i \in \{\pm 1\}, \tau \in \Sym_m\}$ as desired. 
        \item If $F$ is totally real, then, by Lemma \ref{lem: block matrix form of GF}, $G_F \cong \prod_{m}\SO_{2n}$. Since $H$ normalizes $G_F$, $h \in H$ can be represented by $(\ep_1, \dots, \ep_m;\tau;w_1, \dots, w_m)$ where $\ep_i \in \{\pm 1\}, w_i \in W_{\SO_{2n}}, \tau \in \Sym_m$. Since $W_{F}=\prod_m W_{\SO_{2n}}$, we get that $H/W_{F}$ is as desired. 
    \end{enumerate}
\end{proof}

Fix a component $G_{X,l}^{(\sigma)}$ and let 
$B_\sigma\in G_{X,l}^{(\sigma)}$ be the representative. After multiplying $B_\sigma$ by a
suitable scalar in $G_{X,l}^{\circ}$, we may assume that
$B_\sigma\in G$. Since $G_F$ is the similitude-one subgroup of
$G_{X,l}^{\circ}$, it is normalized by $G_{X,l}$, and hence
$B_\sigma\in N_G(G_F)$ and
$
B_\sigma G_F\subseteq G_{X,l}^{(\sigma)}$. By Proposition \ref{prop: quotient map}, we can find $B_{\sigma}' \in N_G(T) \cap N_G(G_F)$ such that $B_{\sigma}'G_F=B_{\sigma}G_F$. 
The class $B_\sigma'T\in H$
determines an element
$B_\sigma'T W_F\in H/W_F$. Let $P_{\sigma}\in \{\pm1\}^m \rtimes \Sym_m$ be the image of $B_\sigma'T W_F$ under the isomorphism in Lemma \ref{lem: H q Wf iso}. Thus
$B_\sigma'T\,W_F=P_\sigma T\,W_F$. Hence, for some $n_\sigma\in N_{G_F}(T)$,
we have $n_\sigma^{-1}B_\sigma'T=P_\sigma T$. Therefore, for some $T_\sigma\in T$, we have $n_\sigma^{-1}B_\sigma'=P_\sigma T_\sigma$. 
Since $n_\sigma\in G_F$ and $B_\sigma'$ normalizes $G_F$,
we have 
$B_\sigma'G_F= B_\sigma'(B_\sigma'^{-1}n_{\sigma}^{-1}B_{\sigma}')G_F=n_\sigma^{-1}B_\sigma'G_F
=P_\sigma T_\sigma G_F$. 

Thus we may choose $B_{\sigma}=P_{\sigma}T_{\sigma}$ where
$P_\sigma\in\{\pm1\}^m\rtimes\Sym_m$ and $T_\sigma\in T$, such that
$B_\sigma G_F\subseteq G_{X,l}^{(\sigma)}$. 

For later use, we describe more explicitly the permutation represented by $P_{\sigma}$. Recall that we have the following commutative diagram:
\[
\begin{tikzcd}
    & \Gal(FL/L) \ni \sigma \arrow[d, two heads, "\kappa"] \arrow[r, hook, "\iota"] & \Gal(F/\Q) \\
    \pi_0(G_{X,l}) \arrow[r, leftarrow, "\simeq"', "\rho_{X,l}"] & \Gal(\Lconn/L).
\end{tikzcd}
\]
For any $\sigma \in \text{Gal}(FL/L)$, its image in $\pi_0(G_{X,l})$ is $\rho_{X,l}(\kappa(\sigma))$, while its image in $\text{Gal}(F/\Q)$ is $\iota(\sigma)$. Recall that $K=F^{\langle \isig \rangle}$, so $d=[F:K]$ is equal to the order of $\iota(\sigma)$ in $\text{Gal}(F/\Q)$.

Notice that $\sigma$ acts on $\Hom(F,\overline{\Q}_l)$.
For $\tau_i \in \Hom(F,\overline{\Q}_l)$, we denote the action of $\sigma$ via 
\[\sigma \cdot \tau_i=\tau_i \circ \iota(\sigma)^{-1}\]This induces the action of the component $B_\sigma$ on the set of $F$-eigenspaces by $
B_\sigma(V_l^{\tau_i})=V_l^{\sigma \cdot \tau_i}$.
Since $F/\Q$ is Galois and $\iota(\sigma)$ has order $d$, every orbit of this action
has length $d$. Hence, when viewed as a permutation in $\Sym_{2mn}$, $P_\sigma$ has cycle type $(\underbrace{d, \dots, d}_{2mn/d}).$


We now describe how to obtain an element of
$(\ep_1, \dots, \ep_m; \pi) \in \{\pm1\}^m\rtimes \Sym_m$ from $\sigma$.

\begin{itemize}
\item When $F$ is CM, write
$\sigma_i=\tau_i|_{F_0}$. Set $\sigma_0:=\isig|_{F_0}$ and $e:=\text{ord}(\sigma_0)$.
The induced permutation $\pi\in \Sym_m$ is determined by
$\sigma_{\pi(i)}=\sigma_i\circ\sigma_0^{-1}$. Since $F_0/\Q$ is Galois, every cycle of $\pi$ has length $e$, so the
cycle type of $\pi$ is
$(\underbrace{e,\ldots,e}_{m/e\text{ times}})$.

We now define $\ep_i$. For each $i$, either $
\sigma \cdot \tau_i=\tau_{\pi(i)}$ or $\tau_{\pi(i)}^*$.
We define
$$
\varepsilon_i=
\begin{cases}
1,& \text{ if } \sigma\cdot \tau_i =\tau_{\pi(i)},\\
-1,&\text{ if } \sigma \cdot \tau_i =\tau_{\pi(i)}^*.
\end{cases}
$$
Thus $P_\sigma$ corresponds to the signed permutation $(\varepsilon_1,\ldots,\varepsilon_m;\pi)
\in\{\pm1\}^m\rtimes \Sym_m$.

 Notice that $e=d$ or $\frac{d}{2}$. If $e=d$, then the $\sigma$-orbit of $\tau_i$ and $\tau_i^*$ are disjoint and each has length $d$; if $e=\frac{d}{2}$, then we have $\isig^e=c$ where $c$ is the complex conjugation, so the $\sigma$-orbit of $\tau_i$ is $\{\sigma^j \cdot \tau_i : 0 \leq j < e\}\sqcup \{\sigma^j \cdot \tau_i^* : 0 \leq j < e\} $, which is a union of two size $e$ subsets.

\item When $F$ is totally real, we have $F_0=F$, so $e=d$.
The permutation $\pi\in\Sym_m$ is similarly determined by
$\tau_{\pi(i)}=\sigma \cdot \tau_i$ for $\tau_i \in \Hom(F, \Qlb)$
and has cycle type
$(\underbrace{d,\ldots,d}_{m/d\text{ times}})$.
For each $i$, the restriction of $P_\sigma$ gives an orthogonal
isomorphism $P_\sigma|_{V_l^{\tau_i}}:
V_l^{\tau_i}\longrightarrow V_l^{\tau_{\pi(i)}}$.
With respect to the chosen compatible orthogonal bases, we define
$$
\varepsilon_i
:=
\det\left(P_\sigma|_{V_l^{\tau_i}}\right)
\in\{\pm1\}.
$$
Thus $P_\sigma$ corresponds to the signed permutation
$(\varepsilon_1,\ldots,\varepsilon_m;\pi)
\in\{\pm1\}^m\rtimes\Sym_m$. The $\sigma$-orbit of $\tau_i$ has length $d$, just as in the case where $F$ is CM and $e=d$.
\end{itemize}
\begin{proposition}\label{prop: bsigma non constant}
For any $\sigma\in\Gal(FL/L)$, the function
$$
b_\sigma=
\Tr\left(
-\,\middle|\,
\wedge^d H^2(X_{\overline L},\Q_l)\otimes\chi^{-d}
\right):
B_\sigma G_F\longrightarrow\Q_l
$$
is non-constant.
\end{proposition}

\begin{proof}
Since $T_\sigma\in T\subset G_F$, replacing $A$ by $T_\sigma A$
allows us to absorb $T_\sigma$. It therefore suffices to find a
one-parameter family $A(t)\in G_F$ such that, for $M(t):=P_\sigma A(t)$, the function
$$
b_\sigma(M(t))
=
\Tr\left(
M(t)\,\middle|\,
\wedge^d H^2(X_{\overline L},\Q_l)
\right)
$$
is nonconstant in $t$. Since we are working on the similitude-one
subgroup, we suppress the factor $\chi^{-d}$.
\begin{enumerate}
\item Suppose $F$ is CM. By Lemma~\ref{lem: block matrix form of GF},
$$
A=
\diag(A_1,A_1^*,\ldots,A_m,A_m^*),
\qquad
A_i^*=J_i^{-1}(A_i^{-1})^\top J_i.
$$
where each $A_i$ is of size $n\times n$. We take
$A(t)=
\diag(A_1(t),A_1(t)^*,\text{I}_n,\text{I}_n,\ldots,\text{I}_n,\text{I}_n)$.
Let
$$
p_k(t):=
\Tr\left(M(t)^k\,\middle|\,V_l\right), \qquad E_j(t):=
\Tr\left(M(t)\,\middle|\,\wedge^j V_l\right)
$$
Since $M(t)^k: V_{l}^{\tau_i} \to V_l^{\sigma^k \cdot \tau_i}$ and the $\sigma$-orbit of $\tau_i$ has length $d$, we have that $p_k(t)=0$ for $1\leq k<d$.
By Newton's identities, we have $jE_j=\sum_{k=1}^j (-1)^{k-1}E_{j-k}p_k$. 
Hence we have $E_j(t)=0$ for $1 \leq j < d$
and $E_d(t)=\frac{(-1)^{d-1}}{d}\,p_d(t)$.

Now recall that $H^2(X_{\overline L},\Q_l)=V_l\oplus\NS_l$, where $\NS_l:=\NS(X_{\overline L})\otimes\Q_l$, and $M(t)$ on $\NS_l$ is constant. Therefore,
\begin{equation*}
\begin{split}
\Tr\left(M(t)\,\middle|\,\wedge^d H^2\right) &= \sum_{j=0}^d \Tr\left(M(t)\,\middle|\,\wedge^j V_l\right) \Tr\left(M(t)\,\middle|\,\wedge^{d-j} \NS_l\right)  \\
&=\Tr\left(M(t)\,\middle|\,\wedge^d V_l\right)+\Tr\left(M(t)\,\middle|\,\wedge^d \NS_l\right)\\
&=k_\sigma + \frac{(-1)^{d-1}}{d}\,p_d(t)
\end{split}
\end{equation*}
where $k_\sigma$ is the contribution from
$\wedge^d\NS_l$ and is independent of $t$. Thus it remains to show that
$p_d(t)=\Tr\left(M(t)^d\,\middle|\,V_l\right)$ is nonconstant.

\begin{enumerate}
\item Suppose first that $d=e$. Then recall that the $\sigma$-orbits of $\tau_1$ and $\tau_1^*$ are disjoint and each has length $d$. Therefore, 
\[M(t)^d \mid_{V_{l}^{\tau_i}}=
\begin{cases}
A_1 & \text{ if $\tau_i$ is in the $\sigma$-orbit of $\tau_1$}\\
A_1^* & \text{ if $\tau_i$ is in the $\sigma$-orbit of $\tau_1^*$}\\
I_n & \text{ otherwise}
\end{cases}\]
Hence $p_d(t)=d \left( \Tr(A_1(t))+\Tr(A_1^*(t))\right)+(2m-2d)n$. Take $A_1(t)=\diag(t,1,\ldots,1)$. Since $J_1$ is diagonal, we have
$A_1(t)^*=\diag(t^{-1},1,\ldots,1)$. Therefore,
$$
p_d(t)=2mn+d(t+t^{-1}-2),
$$
which is non-constant in $t$.

\item Suppose $d=2e$. Then $\tau_i$ and $\tau_i^*$ are in the same $\sigma$-orbit of length $d$. Therefore,
\[
M(t)^d\big|_{V_l^{\tau_i}}
=
\begin{cases}
A_1A_1^*
& \text{if $\tau_i=\sigma^{j}\cdot\tau_1$ for some $0\leq j< e$,}\\
A_1^*A_1
& \text{if $\tau_i=\sigma^{j}\cdot\tau_1^*$ for some $0\leq j< e$,}\\
I_n
& \text{otherwise.}
\end{cases}
\]

Hence $p_d(t)
=
d\Tr\left(A_1(t)A_1(t)^*\right)
+
(2m-d)n$. Take $A_1(t)
=\begin{pmatrix}
1&t\\
0&1
\end{pmatrix}
\oplus I_{n-2}$. Write $J_1=\diag(j_1,j_2, \dots, j_n)$,
then $\Tr\left(A_1(t)A_1(t)^*\right)
=
n-\frac{j_1}{j_2}t^2$,
and hence
$$
p_d(t)
=
2mn-d\frac{j_1}{j_2}t^2,
$$
which is nonconstant in $t$.

\end{enumerate}

\item Suppose $F$ is totally real. By Lemma~\ref{lem: block matrix form of GF},
$$
A=
\diag(A_1,\ldots,A_m),
\qquad
A_i\in \SO_{2n}.
$$
We take $A(t)=\diag(A_1(t),I_{2n},\ldots,I_{2n})$.
The action of $P_\sigma$ on the set of character spaces has cycles of
length $d$, so the same argument as above gives
$\Tr\left(M(t)\,\middle|\,\wedge^d H^2\right)
=
k_\sigma+
\frac{(-1)^{d-1}}{d}\,p_d(t)
$. By the same argument as the case where $d=e$, we have $M(t)^d \mid_{V_l^{\tau_i}}=A_1$ when $\tau_i$ is in the $\sigma$-orbit of $\tau_1$ and is $\text{I}_{2n}$ otherwise. Hence, 
$p_d(t)=d\Tr(A_1(t))+2n(m-d)$. Take 
$A_1(t)=\diag(t)\oplus \text{I}_{n-1}\oplus \diag(t^{-1}) \oplus \text{I}_{n-1}\in\SO_{2n}$.
Then
$$
p_d(t)
=
2mn+d(t+t^{-1}-2),
$$
which is nonconstant in $t$.

\end{enumerate}
Therefore $b_\sigma$ is non-constant on $B_\sigma G_F$.
\end{proof}

Propositions~\ref{prop:density-bound} and \ref{prop: bsigma non constant} imply the following theorem:

\begin{customthm}{\ref{Thm: theorem 2}}
Let $X$ be a K3 surface over a number field $L$ satisfying $(A1), (A2)$ and $(A3)$. Then $\rho(S[\nu])=0$. That is, the set of primes of $L$ with $\mu$-ordinary reduction has density 1.
    
\end{customthm}

\begin{customcor}{\ref{cor:1}}
Let $X$ be a $K3$ surface over a number field $L$ satisfying $(A1), (A2)$ and $(A3)$. Then, the set of primes of $L$ with ordinary reduction has density $\frac{1}{[FL:L]}$.
\end{customcor}

\begin{proof}
Because the $\mu$-ordinary polygon is ordinary if and only if $p$ is totally split in $F/\Q$, the set of ordinary primes of $L$ is the set of primes $\mathfrak{p}$ such that $\iota(\Fr_{\pp}|_{FL})$ is the identity element in $\text{Gal}(F/\Q)$. This set has density $\frac{1}{[FL:L]}$ by Chebotarev's density theorem.

\end{proof}

To show that the assumptions of Theorem~\ref{Thm: theorem 2} are non-vacuous, we exhibit an explicit one-dimensional family of K3 surfaces whose generic members satisfy assumptions $(A1), (A2)$ and $(A3)$.

\begin{example}\label{ex: explicit K3 family}
Consider the $1$-dimensional family of curves $C_{\alpha}$ together with the fixed cover $D$:
\begin{align*}
\begin{split}
C_\alpha: y^5&=x(x-1)(x-\alpha),\qquad \alpha\neq 0,1\\
D: z^5&=t(t-1) \\
\end{split}
\end{align*}
Both $C_{\alpha}$ and $D$ are $\mu_5$-covers of $\Po$ and hence admit the natural $\mu_5$-action.  Notice that $C_{\alpha}$ is Moonen's special family $M[11]$ in \cite[Table 1]{moonen2010special}.

Consider the $\mu_5$-action on $C_{\alpha} \times D$ given by
$$
\zeta_5\cdot(x,y,t,z)=(x,\zeta_5y,t,\zeta_5^{-1}z).
$$
Consider the GIT quotient $S_\alpha :=(C_\alpha\times D)/\mu_5$,  and let $X_\alpha$ be the minimal model of the resolution of $S_\alpha$. The surface $S_\alpha$ has only cyclic quotient singularities. By the standard Hirzebruch--Jung resolution computation in \cite[Section 1.2]{pol}, one obtains $K_{X_\alpha}^2=0$. Furthermore, by the signature formula in \cite[Equation (2.4)]{li2019newton} and the K\"unneth formula, one obtains $h^{2,0}(X_\alpha)=1$ and $h^{1,0}(X_\alpha)=0$. By the classification of product-quotient K3 surfaces in
\cite[Theorem 5.1 and Table 1]{garbagnati2015k3}, we see that
$X_\alpha$ is a K3 surface.

Fix a number field $L$ and choose $\alpha\in L$ such that $C_{\alpha}$ is generic. That is, $J(C_{\alpha})$ has no extra endomorphisms beyond those induced by the natural action of $\mu_5$. We now verify that $X_{\alpha}$ satisfies assumptions $(A1)$--$(A3)$.

Since $C_{\alpha}$ is generic, we have 
$F:=\EE(X_\alpha)=\Q(\zeta_5)$, 
so $F/\Q$ is abelian and $(A1)$ holds. Since $F=\Q(\zeta_5)$ is a CM field and the generic member $X_\alpha$ is non-CM, $(A2)$ holds.

It remains to verify $(A3)$. For any $g\in\Gal(\overline L/FL)$, the $F$-action is defined over $FL$, so $\rho_{X_\alpha,l}(g)$ commutes with the image of $F$. Hence $\rho_{X_\alpha,l}(g)\in\operatorname{Cent}_{\GO(V_l,\Psi_l)}(F)=\GO_F(V_l,\Psi_l)$.

Since $F$ is CM, by the discussion at the start of Section~\ref{sec: orthogonal group},
$$
\GO_F(V_l,\Psi_l)\otimes_{\Q_l}\overline{\Q}_l
\simeq
\mathbb{G}_m\times\prod_{i=1}^m\GL_n.
$$
Thus $\GO_F(V_l,\Psi_l)$ is connected, and by Corollary~\ref{anothergroup},
$\GO_F(V_l,\Psi_l)=G_{X_\alpha,l}^{\circ}$. Therefore, $$\rho_{X_\alpha,l}\bigl(\Gal(\overline L/FL)\bigr)\subseteq G_{X_\alpha,l}^{\circ},$$ 
so $\Lconn\subseteq FL$. Hence $(A3)$ holds.

For a generic K3 surface in the family, Theorem~\ref{Thm: theorem 2} implies that the density of primes of $L$ at which $X_\alpha$ has $\mu$-ordinary reduction is $1$. By Corollary~\ref{cor:1}, the density of ordinary primes is
$$
\frac{1}{[FL:L]}
=
\begin{cases}
1, & L \supseteq \Q(\zeta_5),\\
\frac12, & L \cap \Q(\zeta_5)=\Q(\sqrt5),\\
\frac14, &  L \cap \Q(\zeta_5)=\Q.
\end{cases}
$$
\end{example}

\textbf{Acknowledgment}: This research is funded by Vietnam National University, Ho Chi Minh City under grant number \textbf{DS2024-18-03}. We would like to thank Elena Mantovan for helpful discussions that motivated this research. 
\bibliographystyle{plain}
\bibliography{example.bib}

\end{document}